\documentclass{amsart}
\usepackage[utf8]{inputenc}
\usepackage[english]{babel}
\usepackage{multicol} %
\usepackage{csquotes}
\usepackage[dvipsnames]{xcolor}
\usepackage{indentfirst} 
\usepackage{amsmath} 
\usepackage{amsthm} 
\usepackage{amssymb} 
\usepackage{hyperref} 
\usepackage{mathalpha} 
\usepackage{biblatex} 
\usepackage{mathdots} 
\usepackage{graphicx} 
\usepackage[normalem]{ulem} 

\usepackage{wasysym}
\usepackage{xcolor}
\usepackage{soul}

\newtheorem{prop}{Proposition}
\newtheorem{theorem}{Theorem}

\newtheorem{definition}{Definition}

\newtheorem{remark}{Remark}

\definecolor{firebrick4}{rgb}{0.84,0.1,0.35}
\definecolor{lapislazuli}{rgb}{0.15, 0.38, 0.61}

\usepackage{subfiles}

\title{On Regular Higher Power Rational Diophantine Triples}
\author{Alen Andrašek}
\date{April 2026}

\begin{document}
\begin{abstract}
    A rational Diophantine $m$-tuple is a set $\{a_1,\ldots,a_m\}$ of distinct nonzero rational numbers such that $a_i a_j+1$ is a square for all $1\leq i < j\leq m$. Simi\-larly, we may ask when $a_ia_j+1$ is a $k$-th power. Here, we study the case $k=4$ and produce some non-trivial infinite families of such triples. We show that there are infinitely many triples with positive elements for $k=4$. We also briefly consider the $k=6$ (sextic) and $k=8$ (octic) cases, explaining the difficulties in extending the method to higher exponents.

\end{abstract}

\maketitle

\section{Introduction}
A rational Diophantine $m$-tuple is a set $\{a_1,\ldots,a_m\}$ of distinct nonzero rational numbers such that $a_i a_j+1$ is a square for all $1\leq i < j\leq m$. Sets with this property were first studied by Diophantus (see, e.g. \cite{diophantusOriginal}) who found some particular examples such as 
\begin{align*}
    \left\{\frac{1}{16}, \frac{33}{16}, \frac{17}{4}, \frac{105}{16}\right\}.
\end{align*} Fermat found the integer quadruple $\{1,3,8,120\}$ and Euler extended it to a rational quintuple 
\begin{align*}
    \left\{1,3,8,120, \frac{777480}{8288641}\right\}.
\end{align*}
Stoll \cite{Stoll_2019} proved that this extension was unique. Dujella \cite{dujella2004there} proved that the size $m$ of an integer Diophantine $m$-tuple is at most $5$ and it was proved that an integer Diophantine quintuple does not exist (see \cite{he2018diophantinequintuple}). On the other hand, there exist one-parameter infinite families for rational Diophantine sextuples (see \cite{dujella2024diophantine}, \cite{dujella2017more}, \cite{dujella2017there}). No rational Diophantine septuple is known and no upper bounds on the size $m$ of a rational Diophantine $m$-tuple have been given. For more background on Diophantine $m$-tuples, see \cite{dujella2024diophantine}.

A $k$-th power rational Diophantine $m$-tuple is a set $\{a_1,\ldots,a_m\}$ of distinct nonzero rationals such that $a_ia_j+1 = t_{ij}^k$ for some rational $t_{ij}$. Dujella and Bugeaud \cite{bugeaud2003problem} give the following examples: in the case $k=3$, an integer triple is given by
\begin{align*}
    \{2,171,25326\}
\end{align*}
and in the case $k=4$ by
\begin{align*}
    \{1352, 9539880, 9768370\}.
\end{align*}

Batta, Szikszai and Tengely showed in \cite{batta2025higher} that there exists a rational $k$-th power Diophantine triple for any positive integer $k$ by giving explicit parametric families. For \emph{even} $k$, the set
\begin{align} \label{BST_parametrization}
\left\{    -\frac{1}{r^{\frac{k}{2}}},\frac{r^k-1}{r^{\frac{k}{2}}}, r^\frac{k}{2} \right\}
\end{align}
is a Diophantine triple over $\mathbb{Q}(r)$. Similarly, they also give parametric families for \emph{odd} $k$. In the case $k=3$, they show that every cubic Diophantine pair except $\{-1,1\}$ and $\{-3,3\}$ can be extended to a Diophantine triple (this was also shown by Byeon and Fuchs \cite{byeon2026cube}). They also exhibited a parametric family of third power Diophantine quadruples. On the other hand, Dujella and Bugeaud in \cite{bugeaud2003problem} showed that the size $m$ of a $k$-th power \emph{integer} $m$-tuple is bounded for any $k$ but again no bound is known in the rational case.

In this work, we will mainly study the case $k=4$ (which we will refer to as the \emph{quartic} case). We introduce the notion of a \emph{regular} higher power triple (Definition \ref{def_regularity}) in complete analogy with the classical definition of regularity (see \cite{dujella2024diophantine}). We show that regular quartic triples correspond to points on a surface
\begin{align*}
    \frac{s^2r^2-1}{s^2-r^2}=t^2,
\end{align*}
which we interpret as an elliptic curve with a parameter. By finding parametric points of infinite order, we find some parametric families of triples.

This is done in a couple of ways. In all cases, the points are constructed by requiring an additional condition on $r,s$. From these we find rational parametrizations for $r,s$. To the author's knowledge, no parametric solutions were known, except for $(\ref{BST_parametrization})$. Using the new expressions, we show that there are infinitely many triples with all elements positive (the above parametric family (\ref{BST_parametrization}) always has one negative element). 

Finally, we comment on ideas of how to generalize these ideas to $2k$-th powers, in particular \emph{sextic} (i.e. $6$-th power) triples. This leads to the system 
\begin{align*}
    X^k+Y^k&=Z^k+W^k, \\ 
    XYZW&=\square,
\end{align*}
where $X,Y,Z,W$ are rational numbers.
We use the notation $\square$ to denote an unspecified rational square.

\section{Quartic Regular Triples}
First, let us fix some notation with the following definition. The notation is in complete analogy with the usual notation for Diophantine triples.  
\begin{definition}\label{def_quartic_triple}
The triple $\{a,b,c\}$ of distinct nonzero rationals is a \emph{quartic} Diophantine triple if
\begin{align} \label{def_quartic_triple_equation}
\begin{split}
    ab+1&=r^4, \\
    ac+1&=s^4, \\
    bc+1&=t^4,
\end{split}
\end{align}
for some rational $r,s,t$.
\end{definition}

We want to find rational $a,b,c$ that satisfy such equations. Immediately, we can satisfy the first two equations by defining $b,c$ so that
\begin{align} \label{b,c_parametrization}
\begin{split}
    b &= \frac{r^4-1}{a}, \\
    c &= \frac{s^4-1}{a}.
\end{split}
\end{align}
This leaves us with the problem of determining when the third equation in (\ref{def_quartic_triple_equation}) is satisfied. To simplify things, we introduce an analogue to the concept of a \emph{regular} Diophantine triple, in which case the third equation reduces to that of an elliptic curve over $\mathbb{Q}(r)$.

\begin{definition} \label{def_regularity}
    With notation as in Definition 
 \ref{def_quartic_triple}, a quartic Diophantine triple is called \emph{regular} if $a=s^2-r^2$ (or equivalently if $c = a+b+2r^2$).
\end{definition}

To find a triple $\{a,b,c\}$ such that it is a quartic Diophantine triple, express $b,c$ as in (\ref{b,c_parametrization}) and plug into the third equation to find that,
\begin{align*}
    t^4 = bc+1&= \frac{r^4-1}{a}\frac{s^4-1}{a}+1\\
        &= \frac{(r^4-1)(s^4-1)+a^2}{a^2}.
\end{align*}
Now we can use regularity,
\begin{align*}
    t^4 = bc+1 &= \frac{(r^4-1)(s^4-1)+(s^2-r^2)^2}{(s^2-r^2)^2}\\
    &= \frac{(s^2r^2-1)^2}{(s^2-r^2)^2}.
\end{align*}
Taking square roots, we are left to determine when
\begin{align*}
    \frac{s^2r^2-1}{s^2-r^2} = \pm t^2.
\end{align*}
Without loss of generality, we may assume the plus sign, for a minus sign would just amount to switching $r$ and $s$.
\begin{prop} \label{abc_definition}
Suppose $r,s$ are rationals such that $s,r\neq \pm 1$, $s\neq \pm r$ and
\begin{align} \label{reg_quartic_eqn}
    \frac{s^2r^2-1}{s^2-r^2}&= t^2.
\end{align}
Then $(a,b,c)$ is a rational quartic Diophantine triple with $a=s^2-r^2$,\\ $b=\frac{r^4-1}{a}=t^2-r^2$, $c= \frac{s^4-1}{a} = s^2+t^2$.
\end{prop}
\begin{proof}
We have already proved that this is a sufficient condition for a quartic Diophantine triple to exist. Notice that (\ref{reg_quartic_eqn}) can be rewritten as
    \begin{align*}
        r^2s^2+r^2t^2=1+s^2t^2.
    \end{align*}
We may rewrite this as
    \begin{align*}
        r^4-1&=(r^2-t^2)(r^2-s^2).
    \end{align*}
    Now $r^4-1=ab$ and $s^2-r^2=a$, so upon dividing by $a$ we get that 
    \begin{align*}
        b = t^2-r^2.
    \end{align*}
    Finally, $c=a+b+2r^2=s^2+t^2$. This proves the expressions for $b,c$.
\end{proof}
    
We shall see that there are infinitely many rational solutions to equation (\ref{reg_quartic_eqn}) and, in fact, there is even an integral solution as indicated by a computer search. Some smaller integer $r,s$ are given by
\begin{align*}
    (r,s) = (337,339), (337,3107), (507,1242).
\end{align*}
These are the only ones with $r<s<10000$. Unfortunately, the third pair does not give an integer Diophantine triple $(a,b,c)$ (the corresponding $t$ is not an integer in that case). The first two pairs give the same regular triple
\begin{align*}
    (a,b,c) = (1352, 9539880, 9768370),
\end{align*}
as already found by Dujella and Bugeaud \cite{bugeaud2003problem}. Thus, the only known integer quartic triple is in fact regular. We also note that the rational parametric triple (\ref{BST_parametrization}) is regular.

\subsection{A Special Case}
Notice that the pair $(r,s)=(337,339)$ satisfies $s=r+2$. More generally, let $s = r + \alpha$. Then (\ref{reg_quartic_eqn}) becomes
\begin{align} \label{step_back}
    \frac{r^2(r+\alpha)^2-1}{\alpha(2r+\alpha)} = t^2.
\end{align}
In the special case that $\alpha=2$ one has cancellation of polynomial factors in the above fraction. Notice that we can write
\begin{align*}
    \frac{r^2(r+\alpha)^2-1}{\alpha(2r+\alpha)} &= \frac{(r^2+\alpha r +1)(r^2+\alpha r-1)}{\alpha(2r+\alpha)} \\
    &= \frac{\left((2r+\alpha)^2 - \alpha^2 + 4\right)\left((2r+\alpha)^2 - \alpha^2 - 4\right)}{16\alpha(2r+\alpha)}.
\end{align*}
We see that the \emph{only} case when cancellation occurs automatically is for $\alpha = \pm 2$. Take $\alpha=2$ (the minus sign just amounts to switching $r$ and $s$). Then we get the elliptic curve\footnote{The minimal model of the curve is $y^2=x^3-2x$. It can also be written as $y^2=2x^4-1$. }
\begin{align} \label{alphaje2}
    t^2 = \frac{1}{4}(r+1)(r^2+2r-1).
\end{align}
This is a rank $1$ curve over $\mathbb{Q}$ and thus we get an infinite sequence of rational solutions. The generators of the Mordell-Weil group are $P=(1,1)$ of infinite order and a torsion point $T=(-1,0)$ of order $2$. Then $3P = (337,3107)$, while all other points either give trivial solutions or non-integer ones (this can be checked with a modern CAS specialized in number theory, e.g. MAGMA).

\begin{remark} \label{kompleksna_rjesenja_4}
Since we are studying fourth powers, it may be natural to extend the problem to $\mathbb{Q}(i)$. We notice that equation (\ref{alphaje2}) over $\mathbb{Q}(i)$ gives an elliptic curve of rank $2$ and by inspecting points on this curve, we found two Gaussian integer quartic Diophantine triples (up to sign change and complex conjugation)
\begin{gather*}
     (28+4i, 42+24i, 140+52i),\\
        (15-10i, -15-10i, 16i).
\end{gather*}
\end{remark}

\subsection{Weierstrass form of the main equation}
Returning to (\ref{step_back}), consider $r$ as a parameter. We may multiply this by the square of the denominator to get
\begin{align*}
    (\alpha r + r^2-1)(\alpha r+r^2+1)\alpha (2r+\alpha) &= (\alpha(2r+\alpha)t)^2 \quad \Big/\cdot r^2\\
    (\alpha r + r^2-1)(\alpha r+r^2+1)r\alpha(r\alpha+2r^2) &=(r\alpha(2r+\alpha)t)^2.
\end{align*}
This is an elliptic curve in $\alpha$. Substitute 
\footnote{We could also substitute $x=\alpha r + r^2$ to get the very nice form $y^2=(x^2-1)(x^2-r^4)$.}
\begin{align*}
y &= r\alpha(2r+\alpha)t, \\
x &= \alpha r,
\end{align*}
to get
\begin{align}
    y^2 = (x+r^2-1)(x+r^2+1)x(x+2r^2)
\end{align}
which we then divide by $x^4$.
Set $y/x^2 = v$ and $\frac{1}{x} = u$ to get
\begin{align*}
    v^2 = (u(r^2-1)+1)(u(r^2+1)+1)(2r^2u+1). 
\end{align*}
Furthermore, set new variables $y=2(r^4-1)r^2v$ and $x = 2r^2(r^4-1) u$ to obtain the Weierstrass form,
\begin{align*}
    y^2 = (x+2r^2(r^2+1))(x^2+2r^2(r^2-1))(x+r^4-1)
\end{align*}
One final translation $x \mapsto x-2r^4$ produces
    \footnote{Other equations for the curve are $y^2 = (x+2(r^2+1))(x-2(r^2-1))(x-(r^2-1)(r^2+1))$.
    There is also the form $y^2 = ((r^2-1)x+2)((r^2+1)x-2)(x-1)$ as well as the elegant
    $y^2 =x(x+(r^2-1)^2)(x+(r^2+1)^2)$.}
\begin{align} \label{Weierstrass}
     E_r  :   y^2 = (x+2r^2)(x-2r^2)(x-r^4-1).
\end{align}
The transformation formula in total is given by
\begin{align} \label{transformation_formula}
   \alpha = \frac{2r(r^4-1)}{x-2r^4}
\end{align}
and then
\begin{align*}
    s = \frac{r(x-2)}{x-2r^4}
\end{align*}
where $x$ comes from (\ref{Weierstrass}).
For example, for $r=337$, one can find a point on (\ref{Weierstrass}) with $x=4372394120642$ such that $\alpha = 2$ and then $s=339$. We tried searching iteratively for $r$ with $E_r(\mathbb{Q})$ of high rank in hopes of more likely finding integral points, but the complexity of rank calculation makes such a search very prohibitive. We used PARI/GP for this and most of the time it could not find generators using the default settings. However, some exceptions lead us to find the third parametric solution presented in this paper.

One can check that adding the torsion points of order $2$ either changes the sign of $s$ or reciprocates it, while adding the order $4$ ones switches the corresponding $s$ and $t$ in (\ref{reg_quartic_eqn}).
The curve has rank $0$ over $\mathbb{Q}(r)$ which also prevents us from finding solutions systematically.

\begin{prop}
    The elliptic curve $E_r$ in (\ref{Weierstrass}) has rank $0$ over $\mathbb{Q}(r)$ and torsion group $\mathbb{Z}/2\mathbb{Z}\times\mathbb{Z}/4\mathbb{Z}$, given by
\begin{gather*}\label{torsion}
    E(\mathbb{Q}(r))_{tors} = \{\mathcal{O}, (-2r^2,0),(2r^2,0),(r^4+1,0),\\
    (2r^4, \pm 2r^2(r^4-1)), (2, \pm2(r^4-1))\}.
\end{gather*}
\end{prop} 
\begin{proof} 
We determine the structure of $E_r$ over the function field $\mathbb{Q}(r)$ by employing the specialization criterion of Gusić and Tadić \cite{GUSIC2015137}. For the specialization value $r=11$, the conditions of \cite[Theorem 1.1]{GUSIC2015137} are satisfied, which ensures that the specialization homomorphism $\sigma_{11} : E(\mathbb{Q}(r)) \to E_{11}(\mathbb{Q})$ is injective. A computation in PARI/GP reveals that $E_{11}(\mathbb{Q})$ has rank $0$ and a torsion group isomorphic to $\mathbb{Z}/2\mathbb{Z} \times \mathbb{Z}/4\mathbb{Z}$. It follows that $E_r$ has rank $0$ over $\mathbb{Q}(r)$, and its torsion subgroup is isomorphic to a subgroup of $\mathbb{Z}/2\mathbb{Z} \times \mathbb{Z}/4\mathbb{Z}$. It is easily checked that the elements listed above are torsion, concluding the proof. 
\end{proof}

\begin{remark}
One may consider the curve (\ref{Weierstrass}) in variables $(r,y)$ over $\mathbb{Q}(x)$. This then leads to a curve with generic rank $1$, isomorphic to (\ref{second_parametrization}) below.
\end{remark}
We now present three infinite parametric families of regular quartic triples. 

\section{First parametric family (where $a$ is a square)}
Consider the equation (\ref{reg_quartic_eqn}) for regular quartic triples.
One way to satisfy this is to require that $a = s^2-r^2$ be a rational square. This is achieved when $s= \frac{u^2+1}{u^2-1}r$. Then (\ref{reg_quartic_eqn}) reduces to $s^2r^2-1 = \square$. After plugging in $s$, we get
\begin{align} \label{s^2-r^2 equation_quartic_form}
    (u^2+1)^2r^4 - (u^2-1)^2 &= \square.
\end{align}

\begin{prop}
    The curve (\ref{s^2-r^2 equation_quartic_form}) has the Weierstrass form \begin{align} \label{s^2-r^2 equation_Weierstrass_form}
        y^2 = x^3 + 4(u^4-1)^2x.
    \end{align}
    This curve has torsion $\mathbb{Z}/2\mathbb{Z} = \{\mathcal{O}, (0,0)\}$ and a point of infinite order $P = (2(u^2+1)(u-1)^2, 4(u^2+1)^2(u-1)^2)$ over $\mathbb{Q}(u)$.\\
    A point on this curve gives a regular triple with $(r,s)=\left(\frac{-y}{2x(u^2+1)}, \frac{-y}{2x(u^2-1)}\right)$.
\end{prop}
\begin{proof} Rewrite (\ref{s^2-r^2 equation_quartic_form}) as
\begin{align*}
    r^4-\alpha^2 = w^2
\end{align*}
where $\alpha = \frac{u^2-1}{u^2+1}$, for some rational $u$. Put $w = W + r^2$ to get
\begin{align*}
    -\alpha^2 = W^2+2Wr^2\Big /\cdot W
\end{align*}
Now, set $z = Wr$ to arrive at the equation
\begin{align*}
    W^3+\alpha^2W = -2z^2
\end{align*}
Finally, multiplying by $8$, and setting $-2W = X$, $4z=Z$, we get the Weierstrass equation
\begin{align*}
    Z^2 =  X^3+4\alpha^2 X.
\end{align*}
Plugging in $\alpha=\frac{u^2-1}{u^2+1}$ and setting $x = (u^2+1)^2 X$ and $y = (u^2+1)^3Z$, we finally get \[y^2=x^3+4(u^4-1)^2x.\]
Combining the transformation formulas, we find that
\begin{align*}
    r = \frac{-y}{2x(u^2+1)}. 
\end{align*}
Then $s=r \frac{u^2+1}{u^2-1} = \frac{-y}{2x(u^2-1)}$. 
\end{proof}

Let us concretely analyze the triples that we obtain in this way.
Plug in the point of infinite order $x = 2(u^2+1)(u-1)^2$, to get
\begin{align*}
    r &= \frac{2  \cdot 4(u^2+1)^2(u-1)^2}{4\cdot 2(u^2+1)(u-1)^2(u^2+1)}\\
    &= 1.
\end{align*} Unfortunately, this corresponds to a triple $\{a,b,c\}$ with $ab=0$, which we exclude by definition.
Adding the torsion point $(0,0)$, we get $r=-1$ with the same issue.
On the other hand,
\begin{align*}
    2P = (4u^2, -4u(u^4+1))
\end{align*}
does correspond to a non-trivial example; it maps back to
\begin{align*}
    r = \frac{u^4+1}{2u(u^2+1)}
\end{align*} and $s=\frac{u^4+1}{2u(u^2-1)}$.
This leads to the parametric family of regular quartic Diophantine triples $(a,b,c)=\left(s^2-r^2,\frac{r^4-1}{s^2-r^2}, \frac{s^4-1}{s^2-r^2}\right)$ or explicitly,
\begin{align} \label{a,b,c parametrizacija 1}
    \begin{split}
    a &= \frac{(u^4 + 1)^2}{(u^4 - 1)^2}, \\
    b &= \frac{(u^8 - 4u^6 - 6u^4 - 4u^2 + 1)(u^8 + 4u^6 + 10u^4 + 4u^2 + 1)(u^2 - 1)^2}{16(u^4 + 1)^2(u^2 + 1)^2u^4},\\
    c &= \frac{(u^8 + 4u^6 - 6u^4 + 4u^2 + 1)(u^8 - 4u^6 + 10u^4 - 4u^2 + 1)(u^2 + 1)^2}{16(u^4 + 1)^2(u^2 -1)^2u^4},
    \end{split}
\end{align}
with $u\neq 0, \pm 1$.
We note that, with the exception of the first factor in $c$,
the octic polynomials in $b,c$ factor over $\mathbb{Q}$ into two quartic ones.
This parametric family of triples can be used to prove the following.
\begin{theorem} \label{inftly many positive triples}
    There are infinitely many quartic Diophantine triples that have all positive elements.
\end{theorem}
\begin{proof} In the infinite family of $\{a,b,c\}$ in (\ref{a,b,c parametrizacija 1}), most of the factorization is positive (i.e. the factors are squared or one sees that all roots are non-real). The exception to this is the following factor from $b$,
\begin{align*}
    u^8 - 4u^6 - 6u^4 - 4u^2 + 1 = (u^4 + 2u^3 + 2u + 1)(u^4 - 2u^3 - 2u + 1).
\end{align*}
One can check that this polynomial has real zeros in
\begin{align*}
     \pm\frac{1+\sqrt{3}\pm\sqrt{2\sqrt{3}}}{2}.
\end{align*}
Calculating this to decimals, we see that choosing any rational $u\neq 0,\pm 1$ \emph{outside} of $$\langle -2.2967,-0.4354\rangle \cup \langle 0.4354, 2.2967\rangle$$ suffices to make all the $a,b,c$ positive.
\end{proof}

For an example, take $u=3$, to find the triple
\begin{align*}
   (a,b,c) = \left(\frac{1681}{1600},\frac{8063044}{3404025}, \frac{62349625}{8714304} \right) .
\end{align*}
\begin{remark}
    We note that if $\{a,b,c\}$ is a quartic Diophantine triple, so is $\{-a,-b,-c\}$. Thus, there are infinitely many triples with all entries negative as well. Finally, choosing $u$ so that $b$ is negative, gives infinitely many mixed signs triples.
\end{remark}


\section{Second parametric family (where $s=\alpha r$)}
In our integer solution, we had that $s^2-r^2=k\square$, for $k=2$, $s=339$, $r=337$ as mentioned above. This motivated us to look for a parametrization of the form
\begin{align*}
    s = \frac{ku^2+1}{ku^2-1}r
\end{align*}
and get a more general version of equation (\ref{s^2-r^2 equation_quartic_form}), namely,
\begin{align} \label{s^2-r^2 = k*kvadrat}
    (ku^2+1)^2r^4-(ku^2-1)^2 = k \square.
\end{align}
A rational point is found when $(r,t)=\left(1,\frac{2u}{ku^2-1}\right)$. This shows that it is an elliptic curve. Using standard transformations formulas, we can transform it to a Weierstrass equation. The details are quite more complicated compared to the case $k=1$ above, so we omit them. One will find the short Weierstrass form to be
\begin{align} \label{s^2-r^2 = k*kvadrat Weierstrass}
    y^2 = x^3+4k^2(k^2u^4-1)^2x.
\end{align}
The generic point of infinite order here is 
\begin{align*}
 P = (4k^2u^2, 4k^2u(k^2u^4+1)).
\end{align*}
Note that this was a \emph{double point} for $k$ a square, but here it is not! Transforming it back yields $r=1$, the trivial solution.
Still, we can again take the point $2P$ which transforms back to
\begin{align*}
r = \frac{(ku^2+1)^4-12k^2u^4}{(ku^2-1)^4-12k^2u^4}
\end{align*}
giving some complicated $a,b,c$ that we omit. 

It is the analogy with the previous case that prompted us to consider the substitution $s = \frac{ku^2+1}{ku^2-1}r$, but notice, in fact, that the expression $\alpha = \frac{ku^2+1}{ku^2-1}$ can be any rational number $\alpha$. Thus, this parametric solution could also have been  obtained more simply by considering the condition $s=\alpha r$. Then the Weierstrass form becomes
\begin{align} \label{second_parametrization}
    y^2 = x^3+4\alpha^2(\alpha^2-1)^2x.
\end{align}
Compare this with (\ref{s^2-r^2 = k*kvadrat Weierstrass}). The point of infinite order in terms of $\alpha$ is $P=(4\alpha^2, 4\alpha^2(\alpha^2+1))$. In terms of $\alpha$, from the point $2P$, we get that
\begin{align*}
    r = -\frac{\alpha^4 + 6\alpha^2 - 3}{3\alpha^4 - 6\alpha^2 - 1}
\end{align*}
and this proves the following.
\begin{prop}
    Define $(a,b,c)$ as in Proposition \ref{abc_definition}. For $r = -\frac{\alpha^4 + 6\alpha^2 - 3}{3\alpha^4 - 6\alpha^2 - 1}$ and $s=r\alpha$, $\alpha\neq \pm 1$ we have the following quartic Diophantine triple,
\begin{align*}
a &= \frac{(\alpha^4 + 6\alpha^2 - 3)^2 (\alpha+1)(\alpha-1)}{(3\alpha^4 - 6\alpha^2 - 1)^2}, \\
b &= \frac{-16(5\alpha^4 - 2\alpha^2 + 1)(\alpha^4 - 2\alpha^2 + 5)(\alpha^2 + 2\alpha - 1)(\alpha^2 - 2\alpha - 1)(\alpha^2 + 1)}{(3\alpha^4 - 6\alpha^2 - 1)^2 (\alpha^4 + 6\alpha^2 - 3)^2}, \\
c &= \frac{\left((\alpha^2+1)^4+16\alpha^2(\alpha^2-1)^2\right) \left((\alpha-1)^4+4\alpha^2 \right) \left((\alpha+1)^4+4\alpha^2\right)(\alpha^2 + 1)}{(3\alpha^4 - 6\alpha^2 - 1)^2 (\alpha^4 + 6\alpha^2 - 3)^2}.
\end{align*}
\end{prop}
Finally, we note that another way to obtain (\ref{second_parametrization}) is to consider the curve (\ref{Weierstrass}) in $(y,r)$ with $x$ as a parameter and finding the Weierstrass form.

\section{Third parametric family (from a Pell equation)}

Returning to (\ref{Weierstrass}), a brute force search was performed with the PARI/GP function \emph{ellrank} to find integer $r$ such that a point of infinite order is found with integer coordinates. This would give better chances of the corresponding $\alpha$ to be an integer, after the transformation formula (\ref{transformation_formula}) is applied (thus finding an integer quartic Diophantine triple).

This seems to be a rarity for bigger $r$ (or at least PARI/GP cannot find such points), but in the range $[10^4,10^5]$, a couple of numbers $r$ distinguished themselves, namely $r = 10864$ and  $r = 40545$. For these $r$, PARI/GP found points that translate to $\alpha = -440480880/18817$ and $\alpha = -3067553610/35113$ respectively, which have a significantly smaller denominator compared to $r$ of similar size. We calculated the ratio with $r$ and found
\begin{align*}
    \frac{18817}{10864} = 1.73205081001\ldots \\
    \frac{35113}{40545} = 0.86602540387\ldots
\end{align*}
These seemed to be very good approximations for $\sqrt{3}$ and $\sqrt{3}/2$. Looking up the numbers in the OEIS shows that these numbers are denominators of even convergents of $\sqrt{3}$.
The reason for this turned out to be, that if $p/r$ is an even numbered convergent of $\sqrt{3}$, then as is well known, $p,r$ satisfy the Pell equation,
\begin{align} \label{sqrt3_pell}
    p^2-3r^2=1.
\end{align}
Since we were searching through integers, we found integer $r$ such that they are a solution for the Pell equation above, but any rational numbers would work as well. 
\begin{prop}
    Let $r\in\mathbb{Q}$, $r\neq 0, \pm 1$ and assume there is a $p\in\mathbb{Q}$ such that $p^2-3r^2=1$. Then $P=(r^2+1, pr(1-r^2))$ is a point of infinite order on the curve 
\begin{align*}
     E_r  :   y^2 = (x+2r^2)(x-2r^2)(x-r^4-1).
\end{align*}
\end{prop}
\begin{proof} Simply plug in $x=r^2+1$,
\begin{align*}
    y^2 &= (r^2+1+2r^2)(r^2+1-2r^2)(r^2+1-r^4-1) \\
    &=(3r^2+1)(1-r^2)(r^2-r^4) \\
    &=(3r^2+1)r^2(1-r^2)^2\\
    &= p^2r^2(1-r^2)^2.
\end{align*} 
We see that $(r^2+1, pr(1-r^2))$ is a point on the curve $E_r(\mathbb{Q})$ and it is also not one of the eight torsion points in (\ref{torsion}) (except for $r=0,\pm 1)$. \end{proof}

\begin{remark}
It is interesting from a computational point of view that for $r=2911$, the PARI/GP function \emph{ellrank} did not find the generator with default settings, even though this $r$ also solves the Pell equation (\ref{sqrt3_pell}) above, meaning there must be a point of infinite order. So we see that if we cannot find small examples, perhaps we may have more luck with larger ones. With a different approach, e.g. using \emph{mwrank}, one may determine the rank computationally.
\end{remark}

With the previous proposition, we can construct a point of infinite order when $r$ satisfies (\ref{sqrt3_pell}). Parametrize this quadratic equation to find
\begin{align*}
    r &= \frac{2u}{3-u^2}, \\
    p &= \frac{3+u^2}{3-u^2},
\end{align*}
for some rational $u$. Then using (\ref{transformation_formula}), we will get a corresponding $\alpha$ and $s$ and thus a quartic Diophantine triple $(a,b,c)$. 
We can combine the above point of infinite order with any of the eight torsion points. These eight points give the following possible $s$,
\begin{align*}
    \pm \frac{u^4-9}{8u^2}, \pm \frac{8u^2}{u^4-9},\pm\frac{2u(u^2-3)}{u^4+2u^2+9},\pm\frac{u^4+2u^2+9}{2u(u^2-3)}.
\end{align*}
For each of these possible $s$, the corresponding $t$ from (\ref{reg_quartic_eqn}) will also be one of these expressions. Due to symmetry, switching $s,t$ or changing signs gives the same parametric family of triples. Thus, only two of them are actually distinct, with $s$ reciprocal in the second one. Defining $(a,b,c)$ as in Proposition \ref{abc_definition}, we have the following two propositions.
\begin{prop}\label{sqrt_3 parametrizacija 1}
Let $(r,s)=\left(\frac{2u}{3-u^2}, \frac{8u^2}{u^4-9}\right)$. Then
\begin{align*}
\begin{split}
    a &=  -\frac{4(u^2-9)(u^2 - 1)u^2}{(u^4 - 9)^2}, \\
b &=  \frac{(u^4 - 2u^2 + 9)(u^2 + 3)^2}{4(u^2 - 3)^2u^2},\\
c &=  \frac{(u^8 + 46u^4 + 81)(u^2 + 9)(u^2 + 1)}{4(u^4 - 9)^2u^2}
\end{split}
\end{align*}
is a quartic Diophantine triple if $u\neq 0, \pm 1, \pm 3$.
\end{prop}

\begin{prop} \label{sqrt_3 parametrizacija 2}
Let $(r,s)=\left(\frac{2u}{3-u^2}, \frac{u^4-9}{8u^2}\right)$. Then
\begin{align*}
\begin{split}
   a &=  \frac{(u^4+2u^2+9)^2(u^2 - 1)(u^2 - 9)}{64(u^2 - 3)^2u^4}, \\
b &=  -\frac{64(u^4 - 2u^2 + 9)u^4}{(u^4 + 2u^2 + 9)^2(u^2 - 3)^2},\\
c &=  \frac{(u^8 + 46u^4 + 81)(u^2 + 9)(u^2 + 1)(u^2 - 3)^2}{64(u^4 + 2u^2 + 9)^2u^4}
\end{split}
\end{align*}
is a quartic Diophantine triple if $u\neq 0, \pm 1, \pm 3$.
\end{prop}

For brevity, the following two factorizations were not included above
\begin{align*}
   u^4+2u^2+9 &=  (u^2+2u+3)(u^2-2u+3), \\
   u^8 + 46u^4 + 81 &= (u^4 + 2u^3 + 2u^2 - 6u + 9)(u^4 - 2u^3 + 2u^2 + 6u + 9).
\end{align*}
We note that the parametric families given in the previous propositions indeed can be distinguished. For example, the parametric family in Proposition \ref{sqrt_3 parametrizacija 2} has $b$ always negative. Thus, it cannot prove Theorem \ref{inftly many positive triples}. On the other hand, Proposition \ref{sqrt_3 parametrizacija 1} can do so completely analogously to the proof of Theorem \ref{inftly many positive triples}. (Indeed, there $b,c$ have all factors positive and one simply chooses $u$ such that the factor $-(u^2-9)(u^2-1)$ of $a$ is positive.) 

\begin{remark}
    We can choose $u$ so that $r$ itself is a square, thus at least one of the right hand side numbers in (\ref{def_quartic_triple}) will be an $8$-th power as well. In fact, setting $r=\frac{2u}{3-u^2}=\square$, this equation transforms into the elliptic curve $y^2=x^3-12x$ of rank $1$ and we get an infinite sequence of such triples. This is interesting because even though \cite{batta2025higher} give a parametric family for any power, the parametrization there always has one of the $r,s,t$ equal to $0$ and here this is not the case. The generator $P$ of the Mordell-Weil gives a trivial case which we disregard. But $2P$ gives $u=-2$ and $r=4$. The corresponding triple from Proposition \ref{sqrt_3 parametrizacija 1} is
    \begin{align*}
        (a,b,c) = \left( \frac{240}{49},
        \frac{833}{16}, \frac{69745}{784}\right).
    \end{align*}
\end{remark}

\section{Generalization to Higher Power Tuples}
Our concept of \emph{regularity} can be extended to any even power Diophantine triple in the following way:
\begin{definition}
    The triple $\{a,b,c\}$ is a rational $2k$-th power Diophantine triple if
\begin{align} \label{def_even_power_triple}
\begin{split}
    ab+1&=r^{2k}, \\
    ac+1&=s^{2k}, \\
    bc+1&=t^{2k},
\end{split}
\end{align}
for some rational $r,s,t$. We will call it regular if $a=s^k-r^k$ or equivalently $c=a+b+2r^k$.
\end{definition}
\begin{remark}
    From the second expression, write $c-a-b=2r^k$. Upon squaring this, we can eliminate $r$ to get the equivalent regularity condition 
    \begin{align} \label{regularity_condition}
        a^2+b^2+c^2-2ab-2bc-2ca=4
    \end{align} which applies equally for any even power Diophantine triple.
\end{remark}
In analogy with the quartic case, let $b=\frac{r^{2k}-1}{a}$, $c=\frac{s^{2k}-1}{a}$ so the first two equations in (\ref{def_even_power_triple}) hold, and assume regularity (i.e. $a=s^k-r^k)$. The third equation will be satisfied if
\begin{align*}
    t^{2k} = bc+1 &= \frac{r^{2k}-1}{a}\frac{s^{2k}-1}{a}+1\\
    &=\left(\frac{r^ks^k-1}{s^k-r^k}\right)^2.
\end{align*}Upon taking square roots, we find that we must have
\begin{align*}
    \frac{r^ks^k-1}{s^k-r^k}=t^k,
\end{align*}
in complete analogy with equation (\ref{reg_quartic_eqn}) which we studied in depth for the quartic case. Multiply this out and rewrite to get 
\begin{align*}
    r^ks^k+r^kt^k=s^kt^k+1.
\end{align*}
If we set $x=rs$, $y=rt$, $z=st$, then $xyz=(rst)^2$. Therefore, if we can find rational $x,y,z$ such that
\begin{align}
\begin{split} \label{xyz_equation}
    x^k+y^k&=z^k+1,\\
    xyz&=\square,
\end{split}
\end{align}
we will have found a regular $2k$-th triple by setting $r=\sqrt{\frac{xy}{z}}$ etc.
\begin{remark}
    We may homogenize the equations (\ref{xyz_equation}) to get
    \begin{align}
    \begin{split}
        X^k+Y^k&=Z^k+W^k,\\
        XYZW&=\square.
    \end{split}
    \end{align}
\end{remark}
    
Thus we are looking for numbers that can be written as a sum of two $k$-th powers in two (non-trivially) different ways (generalized taxicab numbers) \emph{and} we would require that the product of bases of those powers is a square.

Since it is not known whether a number exists such that it is expressible as a sum of two $5$-th powers in two ways, any progress in this direction will be difficult with current methods except for the cases $k=3,4$.

Thus we restrict ourselves here to a short analysis of \emph{sextic} and \emph{octic} Diophantine triples.

\subsection{Sextic Diophantine Triples}
For sextic Diophantine triples, take $k=3$ in (\ref{xyz_equation}), to get
\begin{align} \label{power-sum cubic} 
\begin{split}
    x^3+y^3&=z^3+1,\\
    xyz&=\square.
\end{split}
\end{align}
The solution set to the first equation was already parametrized by Euler (see \cite{dickson1920history2}, pp. 552-554, cf. \cite{Choudhry},\cite{Elkies_complete_cubic_parametr}).  It has been referred to as the \emph{Fermat cubic surface} at times. 
The general parametrization is too unwieldy to yield a (simple) condition when $xyz$ is a square.

From the second condition, $z=q^2 xy$ for some rational $q$ so one can write the two equations in one as 
\begin{align*}
    x^3+y^3=1+q^6 x^3 y^3.
\end{align*}
Generically, this has genus $4$ over $\mathbb{Q}(q)$ (checked with MAGMA).

To make things more manageable, we tried our luck with simpler methods as described up to now. Unfortunately, this did not work, but we record the work done.

\begin{prop}
    Considering $x^3+y^3=z^3+1$ as an elliptic curve over $\mathbb{Q}(z)$, its long Weierstrass form is given by
    \begin{align*}
        Y^2 -9(z^3+1)Y = X^3-27(z^3+1)^2
    \end{align*}
    which has trivial torsion and a point of infinite order $P =(3(z^2-z+1), 9(z^2-z+1))$ over $\mathbb{Q}(z)$.
\end{prop}
The point $P$ unfortunately gives $r=\pm 1$ i.e. one of the $a,b,c$ would have to be $0$. Choosing multiples of $P$ did not help. For example, on the curve above, the point $2P$ maps back to
\begin{align*}
    x &= -\frac{2z^3-1}{z^3-1},\\
    y &= \frac{z(z^3+2)}{z^3-1}.
\end{align*}
If $xyz=(rst)^2$ is to indeed be a rational square, then we would necessarily have
\begin{align*}
    xyz= -\frac{z^2(2z^3+1)(z^3+2)}{(z^3-1)^2} = q^2.
\end{align*}
But this already leads to a hyperelliptic equation, which will then only have finitely many rational solutions, if any. 

Another attempt was the following.
\begin{prop}
    Set $x+y=k$. Considering $x^3+y^3=z^3+1$ as an elliptic curve over $\mathbb{Q}(k)$, its Weierstrass form is given by
    \begin{align*}
    Y^2 = X^3 - \frac{27k^3(k^3-4)}{4},
    \end{align*}
    which has trivial torsion and rank (at least) $2$ over $\mathbb{Q}(k)$. Two independent points of infinite order are given by
    \begin{align*}
        P &= (3k^2-3k, \frac{9}{2}k^2(k-2)),\\
        Q &= (3k^2+6k+9, \frac{9}{2}(k^3+4k^2+6k+6)).
    \end{align*}
\end{prop}
We omit the details, but again, no combination of points of small order gave conditions that would be amenable to producing rational solutions. In fact, all resulting equations seemed to be of higher genus, thus if there were any rational solutions, they should be very rare. 

Still, there are solutions to (\ref{power-sum cubic}). By brute-force, we found the following identities
\begin{align*}
    243^3+1600^3 &= 484^3+1587^3, \\
    78^3 + 2809^3 &= 289^3 + 2808^3.
\end{align*}
The corresponding products are, respectively, squares. These identities can be reordered and then dehomogenized in different ways leading to different triples (with different signs). The triples with all positive elements that one obtains this way are
\begin{align*}
(a,b,c) &= \left(
\frac{28041978419}{287496000}, 
\frac{32882791256000}{318751954191},
\frac{131810257007}{1915864488000}
\right),\\
    (a,b,c) &= \left(
\frac{116256402521}{15260373670464},
\frac{3299834241680237}{503598378816},
\frac{4782288288960}{731432701}
\right).
\end{align*}
The identities above are particularly notable, since we can also write them as
\begin{align} \label{sum_of_two_sectic_forms}
\begin{split}
    40^6 + 3^3\cdot 9^6 &= 22^6+3^3\cdot 23^6, \\
    53^6 + 78^3\cdot 1^6 &= 17^6 + 78^3\cdot 6^6.
\end{split}
\end{align}
This shows a decomposition of a number into the form $x^6+h^3y^6$ in two ways. No number is known that is decomposable in two ways as $x^6+y^6$.

The solutions above were found in multiple ways. A search with the explicit parametrization yielded the first example, but not the second one. Both examples were also found simply by brute-force search. No further examples have been found with a search in the range $0 \leq X,Y,Z,W \leq 30000$.

Going back to (\ref{power-sum cubic}), we note that in our two solutions we can choose, e.g. $x=\square$ (in the two examples above $1600/484$ and $2809/289$ are squares). Thus, assume $x=\chi^2$ and $z=k^2y$ for some rational $k$. Then we find that
\begin{align*}
    y^3 = \frac{k^6-1}{\chi^6-1},
\end{align*}
showing another possible way to view the equation. This also shows that whenever one of $x,y,z$ is a square, we will have an identity as in (\ref{sum_of_two_sectic_forms}).
Finally we note that a few dozen examples of sextic Diophantine triples that are not regular have been found by brute force search.


\subsection{Octic Diophantine Triples}
To find octic regular diophantine triples, we would need to solve the following in rationals,
\begin{align} \label{power-sum quartic} 
    X^4+Y^4 &= Z^4+W^4,\\
    XYZW &= \square.
\end{align}
We found no examples. Trying out Euler's parametric family (\cite{dickson1920history2}, pp. 644-647 for the classical derivation),
\begin{align*}
X &= a^7 + a^5b^2 - 2a^3b^4 + 3a^2b^5 + ab^6, \\
Y &= a^6b - 3a^5b^2 - 2a^4b^3 + a^2b^5 + b^7, \\
Z &= a^7 + a^5b^2 - 2a^3b^4 - 3a^2b^5 + ab^6, \\
W &= a^6b + 3a^5b^2 - 2a^4b^3 + a^2b^5 + b^7.
\end{align*}
The condition $XYZW = \square$ is homogeneous and so we can take $b=1$. Then
\begin{align*}
    XZ &= (a^{12} + 2a^{10} - 3a^8 - 2a^6 + 6a^4 - 13a^2 + 1)a^2, \\
    YW &= a^{12} - 13a^{10} + 6a^8 - 2a^6 - 3a^4 + 2a^2 + 1.
\end{align*}
Notice that the above polynomials in $a$ are reciprocal. If we set $YW=f(a)$, then $XZ=a^2f(1/a)$. Thus, $XYZW$ can be written in terms of $u = a+1/a$ and we get the polynomial
\begin{align*}
    XYZW = v^2 = u^{12} - 23u^{10} + 141u^8 - 266u^6 - 80u^4 + 351u^2 + 324.
\end{align*}
We may further set $u=t^2$ and arrive at
\begin{align*}
    v^2 = t^6 - 23t^5 + 141t^4 - 266t^3 - 80t^2 + 351t + 324.
\end{align*}
In any case, this is a genus $2$ curve and has only finitely many points, if any. Therefore, this is not the right approach to generate many examples.
Further parametric families can be found in a recent paper of Bremner \cite{Bremner_4th_powers} among others, but due to their complexity, they will probably not be of use here.

\section{Further remarks}
Recently, a question on MathOverflow  \cite{mathoverflow_parametric_solutions} asked a very related question. Most of the parametric families mentioned there were already known to the author. The only one not already appearing here is due to Mathoverflow user "GH from MO" (going again back to Euler and Hillyer, see \cite{dickson1920history2}, p. 174) which leads to the family of regular triples with 
\begin{align*}
t &= \frac{(u-1)(3u+1)}{2(3u^2+1)}, \\
s &= \frac{(u+1)(3u-1)}{4u}, \\
r &= \frac{9u^4 + 22u^2 + 1}{2(9u^4 + 6u^3 + 2u^2 - 2u + 1)},
\end{align*}
and as before, $a=s^2-r^2$, $b=t^2-r^2$, $c=s^2+t^2$. We omit the explicit expressions for $a,b,c$, but we note that $c$ is a perfect square in $\mathbb{Q}(u)$. This suggests a similar approach to our first parametric family, except that one would ask that \[s^2+t^2\] is a square. With this condition, (\ref{reg_quartic_eqn}) would transform to the Weierstrass model,
\begin{align*}
    y^2 = x^3-u^2(u^2-1)^2x,
\end{align*}
which is $2$-isogenous to the curve (\ref{second_parametrization}) of our second parametric solution (even though we started with a condition that looks like the condition of the first parametric solution). We simply note that the parametric family given above comes from a particular fibration when $u = \frac{3v^2+v}{v-1}$ when the above curve has an extra point of infinite order that maps to the $s,t$ above. We can further find some other parametric families from this curve, but we will stop here.

\section{Conclusion}
In this paper we demonstrated many non-trivial parametric families of quartic Diophantine triples by assuming the regularity condition in Definition \ref{def_regularity} and further conditions on the right hand side variables $r,s,t$. We proved that there are infinitely many quartic Diophantine triples with positive elements. 

What about higher powers? We explored the sextic and octic cases, highlighting the difficulties in extending these methods to higher exponents. Are there infinitely many sextic Diophantine triples with positive elements? What about the octic case? Is there an octic Diophantine triple that does not come from (\ref{BST_parametrization})? Is there a positive one? Are there infinitely many?

We hope that these infinite families help to better understand the structure of quartic Diophantine triples and facilitate the search for further interesting examples of higher order Diophantine $m$-tuples.

\subsection*{Acknowledgments}
The author would like to thank his mentor Matija Kazalicki for many suggestions and discussions regarding the problem. The author also thanks Andrej Dujella for always being ready to answer any question regarding Diophantine $m$-tuples and number theory in general.

\subsection*{Funding}
The author was supported by the Croatian Science Foundation,
project IP-2022-10-500 (TEBAG) as well as by the institutional projects “BELLGI – Bellman functions, graphs and computability” (100‑038/26) and “TopoPha – Topological Phases” (100‑021/26) of the University of Zagreb Faculty of Civil Engineering, funded by the European Union – NextGenerationEU.

\printbibliography

\end{document}